\documentclass[12pt]{amsart}
\usepackage{amssymb}
\usepackage{amsthm}
\usepackage{enumerate}
\usepackage{graphicx}
\usepackage[all,cmtip]{xy}
 \usepackage{tikz-cd}
\makeatletter
\@namedef{subjclassname@2010}{%
  \textup{2010} Mathematics Subject Classification}
\makeatother
\newtheorem{thm}{Theorem}[section]
\newtheorem{corollary}[thm]{Corollary}
\newtheorem{lemma}[thm]{Lemma}
\newtheorem{proposition}[thm]{Proposition}
\theoremstyle{definition}
\newtheorem{definition}[thm]{Definition}
\newtheorem{remark}[thm]{Remark}
\newtheorem{example}[thm]{Example}
\newtheorem{question}[thm]{Question}
\begin{document}
\baselineskip=15pt
\title{Uniformly $S$-pseudo-injective modules}

\author[M. Adarbeh]{Mohammad Adarbeh $^{(\star)}$}
\address{Department of Mathematics, Birzeit University, Birzeit,  Palestine}
\email{madarbeh@birzeit.edu}
\author[M. Saleh]{Mohammad Saleh }
\address{Department of Mathematics, Birzeit University, Birzeit,  Palestine}
\email{msaleh@birzeit.edu}

\thanks{$^{(\star)}$ Corresponding author}
\date{}

\begin{abstract}
 This paper introduces the notion of uniformly $S$-pseudo-injective ($u$-$S$-pseudo-injective) modules as a generalization of $u$-$S$-injective modules. Let $R$ be a ring and $S$ a multiplicative subset of $R$. An $R$-module $E$ is said to be $u$-$S$-pseudo-injective if for any submodule $K$ of $E$, there is $s\in S$ such that for any $u$-$S$-monomorphism $f:K\to E$, $sf$ can be extended to an endomorphism $g:E\to E$. Several properties of this notion are studied. For example, we show that an $R$-module $M$ is $u$-$S$-quasi-injective if and only if $M\oplus M$ is $u$-$S$-pseudo-injective. Two classes of rings related to the class of $QI$-rings are introduced and characterized.
\end{abstract}



\maketitle

\section{Introduction}
Throughout this paper, all rings are commutative with a nonzero identity, and all modules are unitary. Recall that a subset $S$ of a ring $R$ is called a multiplicative subset of $R$ if $1 \in S$, $0 \notin S$, and $s_1s_2\in S$ for all $s_1,s_2 \in S$. Throughout, $R$ denotes a commutative ring with identity and $S$ a multiplicative subset of $R$. Let $M$, $N$, and $L$ be $R$-modules.
    \begin{enumerate}
    \item[(i)] $M$ is called a $u$-$S$-torsion module if there exists $s \in S$ such that $sM = 0$ \cite{Z}.
    \item[(ii)] An $R$-homomorphism $f: M \to N$ is called a $u$-$S$-monomorphism ($u$-$S$-epimorphism) if $\text{Ker}(f)$ ($\text{Coker}(f)$) is a $u$-$S$-torsion module \cite{Z}.
    \item[(iii)] An $R$-homomorphism $f: M \to N$ is called a $u$-$S$-isomorphism if $f$ is both a $u$-$S$-monomorphism and a $u$-$S$-epimorphism \cite{Z}.
\item[(iv)]  An $R$-sequence $M \xrightarrow{f} N \xrightarrow{g} L$ is said to be $u$-$S$-exact if there exists $s \in S$ such that $s\text{Ker}(g) \subseteq \text{Im}(f)$ and $s\text{Im}(f) \subseteq \text{Ker}(g)$. A $u$-$S$-exact sequence $0 \to M \to N \to L \to 0$ is called a short $u$-$S$-exact sequence \cite{ZQ}. 
\item[(v)] A short $u$-$S$-exact sequence $0\to M \xrightarrow{f} N \xrightarrow{g} L\to 0$ is said to be $u$-$S$-split (with respect to $s$) if there is $s\in S$ and an $R$-homomorphism $f':N \to M$ such that $f'f=s1_{M}$, where $1_{M}:M\to M$ is the identity map on $M$ \cite{ZQ}.
 \end{enumerate}
 The notion of $u$-$S$-injective modules was introduced and studied by W. Qi et al. in \cite{QK}. They defined an $R$-module $E$ to be $u$-$S$-injective if the induced sequence
 \begin{center}
     $0 \to \text{Hom}_{R}(C,E)\to \text{Hom}_R(B,E)\to \text{Hom}_R(A,E) \to 0$
 \end{center}
 is $u$-$S$-exact for any $u$-$S$-exact sequence $0 \to
A\to B\to C \to 0$. Equivalently, if the induced sequence $0 \to \text{Hom}_{R}(C,E)\to \text{Hom}_R(B,E)\to \text{Hom}_R(A,E) \to 0$ is $u$-$S$-exact for any short exact sequence $0 \to
A\to B\to C \to 0$ \cite[Theorem 4.3]{QK}. Injective modules and $u$-$S$-torsion modules are $u$-$S$-injective \cite{QK}. X. L. Zhang and W. Qi \cite{ZQ} introduced the notions of $u$-$S$-semisimple modules and $u$-$S$-semisimple rings. An $R$-module $M$ is called $u$-$S$-semisimple if any short $u$-$S$-exact sequence $0\to A\to M\to C\to 0$ is $u$-$S$-split. A ring $R$ is called $u$-$S$-semisimple if any free $R$-module is $u$-$S$-semisimple. Recently, M. Adarbeh and M. Saleh \cite{MM} introduced and studied the notion of $u$-$S$-injective relative to a module. They defined an $R$-module $E$ to be $u$-$S$-injective relative to a module $M$ if for any $u$-$S$-monomorphism $f:K\to M$, the induced map $\text{Hom}_{R}(f,E): \text{Hom}_R(M,E)\to \text{Hom}_R(K,E)$ is a $u$-$S$-epimorphism. They also introduced the notion of $u$-$S$-quasi-injective modules. An $R$-module $E$ is called $u$-$S$-quasi-injective if it is $u$-$S$-injective relative to $E$. By \cite[Theorem 2.4]{MM}, we conclude that an $R$-module $E$ is $u$-$S$-quasi-injective if and only if for any submodule $K$ of $E$, there is $s\in S$ such that for any $R$-homomorphism $f:K\to E$, $sf$ can be extended to an endomorphism $g:E\to E$. In this paper, we define $u$-$S$-pseudo-injective modules as follows: an $R$-module $E$ is said to be $u$-$S$-pseudo-injective if for any submodule $K$ of $E$, there is $s\in S$ such that for any $u$-$S$-monomorphism $f:K\to E$, $sf$ can be extended to an endomorphism $g:E\to E$. We have 
\begin{center}
$u$-$S$-injective $\Rightarrow$  $u$-$S$-quasi-injective $\Rightarrow$  $u$-$S$-pseudo-injective.
\end{center}

In Section 2, we discuss some properties of $u$-$S$-pseudo-injective modules. For example, we show in Remark \ref{rem1} that if $S\subseteq U(R)$, where $U(R)$ denotes the set of all units of $R$, then the notions of $u$-$S$-pseudo-injective modules and pseudo-injective modules coincide. However, they are different in general (see Example \ref{ex2}). Theorem \ref{thm2} and Corollary \ref{cor1} give the uniformly $S$-version of \cite[Theorem 1]{JS} and its corollary, respectively, in the commutative case. Theorem \ref{thm4} gives a new characterization of $u$-$S$-semisimple rings in terms of $u$-$S$-pseudo-injective modules. 

In Section 3, firstly, we introduce two classes of rings related to the class of $QI$-rings (rings in which every quasi-injective module is injective). Let $S$ be a multiplicative subset of a ring $R$. $R$ is called a $Q$$u$-$S$-$I$-ring ($u$-$S$-$Q$$u$-$S$-$I$-ring) if every quasi-injective $R$-module is $u$-$S$-injective (every $u$-$S$-quasi-injective $R$-module is $u$-$S$-injective). By \cite[the Corollary after Proposition 1]{B} and since every $QI$-ring is an $SSI$-ring (a ring in which every semisimple module is injective), we have every commutative $QI$-ring is semisimple. By Remark \ref{rem2}, we have 
\begin{center}
$QI$-rings $\Rightarrow$ $u$-$S$-semisimple rings $\Rightarrow$ $u$-$S$-$Q$$u$-$S$-$I$-rings $\Rightarrow$ $Q$$u$-$S$-$I$-rings.
\end{center}
 We characterize $Q$$u$-$S$-$I$-rings ($u$-$S$-$Q$$u$-$S$-$I$-rings) in Theorem \ref{thm1} (Theorem \ref{thrm1}). In Proposition \ref{plocal}, we give a local characterization of $QI$-rings. The last result (Theorem \ref{thm3}) of this section gives a characterization of rings in which every $u$-$S$-pseudo-injective module is $u$-$S$-injective. Throughout, $U(R)$ denotes the set of all units of $R$; $\text{Max}(R)$ denotes the set of all maximal ideals of $R$; $\text{Spec}(R)$ denotes the set of all prime ideals of $R$.
\section{$u$-$S$-pseudo-injective modules}\label{d}

We start this section by recalling the following definition from \cite{MM}:

\begin{definition}\label{def1}
   Let $S$ be a multiplicative subset of a ring $R$ and $E,M$ be $R$-modules.
\begin{enumerate}
\item[(i)]  $E$ is said to be $u$-$S$-injective relative to $M$ if for any $u$-$S$-monomorphism $f:K\to M$, the map 
 $$\text{Hom}_{R}(f,E): \text{Hom}_R(M,E)\to \text{Hom}_R(K,E)$$ is a $u$-$S$-epimorphism.
 \item[(ii)] $E$ is said to be $u$-$S$-quasi-injective if it is $u$-$S$-injective relative to $E$. 
\end{enumerate}   
\end{definition}

\begin{lemma}\label{lem1}
  Let $S$ be a multiplicative subset of a ring $R$ and $E$ an $R$-module. Then the following are equivalent:
\begin{enumerate}
    \item[(1)] $E$ is $u$-$S$-quasi-injective; 
    \item[(2)] for any monomorphism $h:K\to E$, there is $s\in S$ such that for any $R$-homomorphism $f:K\to E$, there is $g\in \text{End}_R(E)$ such that $sf=gh$;
    \item[(3)] for any submodule $K$ of $E$, there is $s\in S$ such that for any $R$-homomorphism $f:K\to E$, $sf$ can be extended to $g\in \text{End}_R(E)$.
\end{enumerate}
\end{lemma}

\begin{proof}
    This follows from \cite[Theorem 2.4]{MM}.
\end{proof}

Let $R$ be a ring. Recall that an $R$-module $E$ is called pseudo-injective if for any submodule $K$ of $E$, every monomorphism $f:K\to E$ can be extended to an endomorphism $g:E\to E$ \cite{JS}. Now, we introduce the uniformly $S$-version of pseudo-injective modules. 

\begin{definition}\label{def2}
   Let $S$ be a multiplicative subset of a ring $R$. An $R$-module $E$ is said to be $u$-$S$-pseudo-injective if for any submodule $K$ of $E$, there is $s\in S$ such that for any $u$-$S$-monomorphism $f:K\to E$, $sf$ can be extended to an endomorphism $g:E\to E$.
   \end{definition}

\begin{remark} \label{rem1}
      Let $S$ be a multiplicative subset of a ring $R$. 
      \begin{enumerate}
          \item[(1)] If $S\subseteq U(R)$, the notions of $u$-$S$-pseudo-injective modules and pseudo-injective modules coincide.
          \item[(2)] $u$-$S$-injective $\Rightarrow$  $u$-$S$-quasi-injective $\Rightarrow$  $u$-$S$-pseudo-injective. 
          \item[(3)] By (2) and \cite[ Proposition 3.6]{MM}, every $u$-$S$-semisimple module is $u$-$S$-pseudo-injective. 
      \end{enumerate}
      \begin{center}
    \end{center}
\end{remark}

For an $R$-module $M$, let $K\leq M$ denote that $K$ is a submodule of $M$. The following proposition provides some properties of $u$-$S$-pseudo-injective modules.

\begin{proposition}\label{prop1}
   Let $S$ be a multiplicative subset of a ring $R$. 
   \begin{enumerate}
   \item[(1)] Let $0\to A\xrightarrow{f} B\xrightarrow{g} C\to 0 $ be a $u$-$S$-split $u$-$S$-exact sequence. If $B$ is a $u$-$S$-pseudo-injective module, then so are $A$ and $C$.
    \item[(2)] If $A\oplus B$ is a $u$-$S$-pseudo-injective module, then so are $A$ and $B$. 
  \item[(3)] Let $f:A\to B$ be a $u$-$S$-isomorphism. Then $A$ is $u$-$S$-pseudo-injective if and only if $B$ is $u$-$S$-pseudo-injective. 
  \item[(4)] If $A$ is a $u$-$S$-pseudo-injective module, then any $u$-$S$-monomorphism $f:A\to A$ $u$-$S$-splits.
\end{enumerate}  
\end{proposition}
\begin{proof}  (1) Since $0\to A\xrightarrow{f} B\xrightarrow{g} C\to 0 $ $u$-$S$-splits, there are $R$-homomorphisms $f':B\to A$ and $g':C\to B$ such that $f'f=t1_A$ and $gg'=t1_C$ for some $t\in S$. Suppose that $B$ is a $u$-$S$-pseudo-injective module. Let $K\leq A$. Then $f(K)\leq B$. Since $B$ is $u$-$S$-pseudo-injective, then there is $s\in S$ such that for any $u$-$S$-monomorphism $h:f(K)\to B$, there is $e\in \text{End}_R(B)$ such that $sh=e\mid_{f(K)}$. Let $h':K\to A$ be any $u$-$S$-monomorphism. Since $f'(f(K))=tK\subseteq K$ and $f,h'$ are $u$-$S$-monomorphisms, we have $h:=fh'(f'|_{f(K)}):f(K)\to B$ is a $u$-$S$-monomorphism. So $sh=e\mid_{f(K)}$ for some $e\in \text{End}_R(B)$. Let $s'=st^2$ and $e'=f'ef$. Then $e'\in \text{End}_R(A)$ and for $k\in K$, we have 
       $$e'(k)=f'ef(k)=f'sh(f(k))=sf'fh'f'(f(k))=st^2h'(k)=s'h'(k).$$  
       Hence $s'h'=e'\mid_{K}$. Thus $A$ is a $u$-$S$-pseudo-injective module. Similarly, we can show that $C$ is a $u$-$S$-pseudo-injective module.\\[0.1cm]
      (2) Let $i_A:A \to A\oplus B$ be the natural injection and $p_B: A\oplus B\to B$ be the natural projection. Since $0\to A\xrightarrow{i_A} A\oplus B\xrightarrow{p_B} B\to 0$ is a split exact sequence (hence a $u$-$S$-split $u$-$S$-exact sequence), then this part follows from part (1).  \\[0.1cm]
(3) This follows from part (1) and the fact that the $u$-$S$-exact sequences $0\to 0\to A\xrightarrow{f} B\to 0$ and $0\to A\xrightarrow{f} B\to 0\to 0$ are $u$-$S$-split.\\[0.1cm]
(4) Suppose that $A$ is a $u$-$S$-pseudo-injective module. Let $f:A\to A$ be any $u$-$S$-monomorphism. Then $f:A\to \text{Im}(f)$ is a $u$-$S$-isomorphism. Then by \cite[Lemma 2.1]{ZQ}, there is a $u$-$S$-isomorphism $f':\text{Im}(f)\to A$ and $t\in S$ such that $f'f=t1_A$. Since $\text{Im}(f)\leq A$ and $A$ is $u$-$S$-pseudo-injective, then there is an $R$-endomorphism $g:A\to A$ such that $sf'=g|_{\text{Im}(f)}$ for some $s\in S$. For any $a\in A$, $sta=sf'f(a)=g(f(a))$. Hence $f$ $u$-$S$-splits. \end{proof}

\begin{thm}\label{thm2}
   Let $S$ be a multiplicative subset of a ring $R$. If $A\oplus B$ is a $u$-$S$-pseudo-injective module and $\varphi:A\to B$ is a $u$-$S$-monomorphism, then $\varphi$ $u$-$S$-splits and $A$ is $u$-$S$-quasi-injective.
\end{thm}

\begin{proof}
 Let $T=0\oplus \varphi(A)$ and consider the monomorphism $f:T\to A\oplus B$ given by $f(0,\varphi(a))=(a,0)$ for all $a\in A$. Since $A\oplus B$ is $u$-$S$-pseudo-injective, so there is $g\in \text{End}_{R}(A\oplus B)$ such that $sf=g|_{T}$ for some $s\in S$. Let $i_2:B \to A\oplus B$ be the natural injection, $p_1: A\oplus B\to A$ be the natural projection, and $\psi:=p_1gi_2:B\to A$. Then for $a\in A$, we have $$\psi\varphi(a)=p_1gi_2\varphi(a)=p_1g(0,\varphi(a))=p_1sf(0,\varphi(a))=sp_1(a,0)=sa.$$ So $\psi\varphi=s1_A$ and hence $\varphi$ $u$-$S$-splits. By \cite[Lemma 2.8]{KMOZ}, $B$ is $u$-$S$-isomorphic to $A\oplus C$ for some module $C$. By \cite[Proposition 2.4]{KMOZ}, $A\oplus B$ is $u$-$S$-isomorphic to $A\oplus A \oplus C$. Since $A\oplus B$ is $u$-$S$-pseudo-injective, then so is $A\oplus A \oplus C$ by Proposition \ref{prop1} (3). Again, by Proposition \ref{prop1} (2), we have $A\oplus A$ is $u$-$S$-pseudo-injective. Let $K\leq A$ and $K'=K\oplus 0$. Since $A\oplus A$ is $u$-$S$-pseudo-injective, there is $t\in S$ such that for any $u$-$S$-monomorphism $h':K'\to A\oplus A$, there is $g'\in \text{End}_R(A\oplus A)$ such that $th'=g'|_{K'}$. Let $h:K\to A$ be any $R$-homomorphism. Then $h':K'\to A\oplus A$ given by $h'(x,0)=(x,h(x))$, $x\in K$, is a monomorphism. So $th'=g'|_{K'}$ for some $g'\in \text{End}_R(A\oplus A)$. Let $q:A\to A\oplus A$ be the map $x\mapsto (x,0)$, $x\in A$ and $p:A\oplus A\to A$ be the map $(x,y)\mapsto y$, $x,y\in A$. Then $g:=pg'q\in  \text{End}_R(A)$ and for $k\in K$, we have $g(k)=pg'q(k)=pg'(k,0)=pth'(k,0)=tp(k,h(k))=th(k)$. Hence $th=g|_{K}$. Thus, for any $K\leq A$, there is $t\in S$ such that for any $R$-homomorphism $h:K\to A$, $th$ can be extended to $g\in \text{End}_R(A)$. Therefore, by Lemma \ref{lem1}, $A$ is $u$-$S$-quasi-injective.
\end{proof}

\begin{corollary}\label{cor1}
     Let $S$ be a multiplicative subset of a ring $R$ and $M$ an $R$-module. Then $M$ is $u$-$S$-quasi-injective if and only if $M\oplus M$ is $u$-$S$-pseudo-injective. 
     \end{corollary}

\begin{proof}
   Suppose that $M$ is $u$-$S$-quasi-injective. Then by \cite[Proposition 3.8 ]{MM}, $M\oplus M$ is $u$-$S$-quasi-injective and hence $M\oplus M$ is $u$-$S$-pseudo-injective by Remark \ref{rem1} (2). The converse follows from Theorem \ref{thm2}.   
\end{proof}

Let $M$ be an $R$-module. For a positive integer $n$, let $M^{(n)}=\underbrace{M\oplus M\oplus\cdots \oplus M}_{n\text{-times}}$.

\begin{corollary}\label{cor2}
  Let $S$ be a multiplicative subset of a ring $R$ and $M$ an $R$-module. For any integer $n\geq 2$, $M$ is $u$-$S$-quasi-injective if and only if $M^{(n)}$ is $u$-$S$-pseudo-injective. 
\end{corollary}

\begin{proof}
 ($\Rightarrow$). Since $M$ is $u$-$S$-quasi-injective, $M$ is $u$-$S$-injective relative to $M$. So by \cite[Proposition 3.8]{MM}, $M^{(n)}$ is $u$-$S$-quasi-injective and hence by Remark \ref{rem1} (2), $M^{(n)}$ is $u$-$S$-pseudo-injective. \\
  ($\Leftarrow$). For $n=2$, apply Corollary \ref{cor1}. For $n>2$, since $M^{(2)}\oplus M^{(n-2)}\cong  M^{(n)}$ is $u$-$S$-pseudo-injective, then by Proposition \ref{prop1} (2), $M^{(2)}$ is $u$-$S$-pseudo-injective and hence by Corollary \ref{cor1}, $M$ is $u$-$S$-quasi-injective. 
\end{proof}

In the following example, we will use Corollary \ref{cor1} to construct an example of a $u$-$S$-pseudo-injective module that is not pseudo-injective.

\begin{example}\label{ex2}
    Let $R=\mathbb{Z}$, $S=R\setminus \{0\}$, and $M=R$. Then by \cite[Example 3.7]{MM}, $M$ is a $u$-$S$-quasi-injective module that is not quasi-injective. By Corollary \ref{cor1}, $M\oplus M$ is a $u$-$S$-pseudo-injective module. However, since $M$ is not quasi-injective, then $M\oplus M$ is not pseudo-injective \cite{JS}.
\end{example}

 Let $\mathfrak{p}$ be a prime ideal of a ring $R$. Then $S=R\setminus\mathfrak{p}$ is a multiplicative subset of $R$. We say that an $R$-module $M$ is $u$-$\mathfrak{p}$-pseudo-injective if $M$ is $u$-$S$-pseudo-injective. Another application of Corollary \ref{cor1} is the following example of a $u$-$S$-pseudo-injective module that is not $u$-$S$-injective.
 
\begin{example}\label{ex3}
    Let $R=\mathbb{Z}$ and $M=\mathbb{Z}_2$. Then by \cite[Example 3.5]{MM}, there is a maximal ideal $\mathfrak{m}$ of $R$ such that $M$ is a $u$-$\mathfrak{m}$-quasi-injective module that is not $u$-$\mathfrak{m}$-injective. By Corollary \ref{cor1}, $M\oplus M$ is a $u$-$\mathfrak{m}$-pseudo-injective module. However, $M\oplus M$ is not $u$-$\mathfrak{m}$-injective by \cite[ Corollary 2.8 (2)]{MM}.
\end{example}

\begin{proposition}\label{proplocal}
Let $R$ be a ring and $M$ an $R$-module. If $M$ is $u$-$\mathfrak{m}$-pseudo-injective for every $\mathfrak{m}\in \text{Max}(R)$, then $M$ is pseudo-injective.
\end{proposition}

\begin{proof}
    Let $K\leq M$ and $f:K\to M$ be a monomorphism. Then $f:K\to M$ is a $u$-$\mathfrak{m}$-monomorphism for every $\mathfrak{m}\in \text{Max}(R)$. By hypothesis, for every $\mathfrak{m}\in \text{Max}(R)$, there is $s_{\mathfrak{m}}\in R\setminus \mathfrak{m}$ and $g_{\mathfrak{m}}\in \text{End}_R(M)$ such that $s_\mathfrak{m}f=g_\mathfrak{m}|_{K}$. Then $I:=\langle\{s_{\mathfrak{m}}\mid \mathfrak{m}\in \text{Max}(R)\}\rangle =R$, for if $I\neq R$, then there is $\mathfrak{m}\in \text{Max}(R)$ such that $I\subseteq \mathfrak{m}$ which implies $s_\mathfrak{m}\in \mathfrak{m}$, a contradiction. So $I=R$. Hence, there are finite sets $\{r_{\mathfrak{m}_i}\}_{i=1}^{n}\subseteq R$ and $\{s_{\mathfrak{m}_i}\}_{i=1}^{n}\subseteq \{s_{\mathfrak{m}}\mid \mathfrak{m}\in \text{Max}(R)\}$ such that $1=\sum\limits_{i=1}^{n}r_{\mathfrak{m}_i}s_{\mathfrak{m}_i}$. Let $g=\sum\limits_{i=1}^{n}r_{\mathfrak{m}_i}g_\mathfrak{m_i}$. Then $g \in \text{End}_R(M)$ and $f=\sum\limits_{i=1}^{n}r_{\mathfrak{m}_i}s_{\mathfrak{m}_i}f=\sum\limits_{i=1}^{n}r_{\mathfrak{m}_i}(g_\mathfrak{m_i}|_{K})=g|_{K}$. Thus $M$ is pseudo-injective.
\end{proof}

The following proposition shows that the converse of Proposition \ref{proplocal} is true if $R$ is a ring in which every pseudo-injective $R$-module is quasi-injective.

\begin{proposition}
Let $R$ be a ring such that every pseudo-injective $R$-module is quasi-injective, and let $M$ be an $R$-module. Then the following statements are equivalent:
    \begin{enumerate}
          \item[(1)] $M$ is pseudo-injective;
          \item[(2)] $M$ is $u$-$\mathfrak{p}$-pseudo-injective for every $\mathfrak{p}\in \text{Spec}(R)$;
          \item[(3)] $M$ is $u$-$\mathfrak{m}$-pseudo-injective for every $\mathfrak{m}\in \text{Max}(R)$.
      \end{enumerate}

\end{proposition}

\begin{proof}
    $(1)\Rightarrow (2)$: Suppose that $M$ is pseudo-injective. Then $M$ is quasi-injective. So by \cite[Remark 3.2 (2)]{MM} and Remark \ref{rem1} (2), $(2)$ holds. \\
     $(2)\Rightarrow (3)$: Clear.\\
    $(3)\Rightarrow (1)$: This follows from Proposition \ref{proplocal}.
    \end{proof}

\begin{proposition}
  Let $S$ be a multiplicative subset of a ring $R$. If $A\oplus B$ is a $u$-$S$-pseudo-injective module, then $A$ is $u$-$S$-injective relative to $B$ and $B$ is $u$-$S$-injective relative to $A$.
\end{proposition}

\begin{proof}
   Assume that $A\oplus B$ is $u$-$S$-pseudo-injective. To show $A$ is $u$-$S$-injective relative to $B$, let $K\leq B$. By hypothesis, there is $s\in S$ such that for any $u$-$S$-monomorphism $h:0\oplus K\to A\oplus B$, there is $g\in \text{End}_R(A\oplus B)$ such that $sh=g|_{0\oplus K}$. Let $f:K\to A$ be any homomorphism. Consider the monomorphism $h:0\oplus K\to A\oplus B$ given by $h(0,k)=(f(k),k)$, $k\in K$. Then $sh=g|_{0\oplus K}$ for some $g\in \text{End}_R(A\oplus B)$. Let $i_2:B\to A\oplus B$ and $p_1:A\oplus B\to A$ be the natural injection and projection, respectively, and let $g'=p_1gi_2:B\to A$. Then for $k\in K$, we have $g'(k)=(p_1gi_2)(k)=p_1g(0,k)=sp_1h(0,k)=sp_1(f(k),k)=sf(k)$. So $sf=g'|_K=g'i_K$, where $i_K:K\to B$ is the inclusion map. Thus, for any $K\leq B$, the map $(i_K)^*:\text{Hom}(B,A)\to \text{Hom}(K,A)$ is a $u$-$S$-epimorphism. Hence, by \cite[Theorem 2.4]{MM}, $A$ is $u$-$S$-injective relative to $B$. Since $B\oplus A\cong A\oplus B$ is $u$-$S$-pseudo-injective, then by above, $B$ is $u$-$S$-injective relative to $A$.
\end{proof}

 The last result of this section gives a new characterization of $u$-$S$-semisimple rings in terms of $u$-$S$-pseudo-injective modules.

\begin{thm}\label{thm4}
Let $S$ be a multiplicative subset of a ring $R$. Then $R$ is $u$-$S$-semisimple if and only if every $R$-module is $u$-$S$-pseudo-injective.    
\end{thm}

\begin{proof}
    Let $R$ be a $u$-$S$-semisimple ring. Then by \cite[Theorem 3.11]{MM}, every $R$-module is $u$-$S$-quasi-injective and hence by Remark \ref{rem1} (2), every $R$-module is $u$-$S$-pseudo-injective. Conversely, let $M$ be any $R$-module. Then by hypothesis, $M\oplus M$ is $u$-$S$-pseudo-injective. Hence, by Corollary \ref{cor1}, $M$ is $u$-$S$-quasi-injective. Thus, every $R$-module is $u$-$S$-quasi-injective. Again by \cite[Theorem 3.11]{MM}, $R$ is $u$-$S$-semisimple.
\end{proof}

\section{$Q$$u$-$S$-$I$-rings and $u$-$S$-$Q$$u$-$S$-$I$-rings}

In this section, we introduce two classes of rings related to the class of $QI$-rings. Recall that a ring $R$ is called a $QI$-ring if every quasi-injective $R$-module is injective \cite{JLS}.

\begin{definition}
   Let $R$ be a ring and $S$ a multiplicative subset of $R$. We say that
   
   \begin{enumerate}
       \item[(1)] $R$ is a $Q$$u$-$S$-$I$-ring if every quasi-injective $R$-module is $u$-$S$-injective.
       \item[(2)] $R$ is a $u$-$S$-$Q$$u$-$S$-$I$-ring if every $u$-$S$-quasi-injective $R$-module is $u$-$S$-injective.
   \end{enumerate}
\end{definition}

Recall that a ring $R$ is called an $SSI$-ring if every semisimple $R$-module is injective \cite{B}.

\begin{remark}\label{rem2}
We have the following implications:
\[
\begin{array}
[c]{cccccccc}
& & \text{semisimple rings} &  & \Leftrightarrow & & \text{$QI$-rings } &\\
& & \Downarrow  & &  && \Downarrow & \\
 & & \text{$u$-$S$-semisimple rings } & \Rightarrow & \text{$u$-$S$-$Q$$u$-$S$-$I$-rings }&\Rightarrow & \text{$Q$$u$-$S$-$I$-rings} &
\end{array}
\] 
\end{remark}

\begin{proof}
 Clearly, every semisimple ring is both a $QI$-ring and a $u$-$S$-semisimple ring. By \cite[the Corollary after Proposition 1]{B} and since every $QI$-ring is an $SSI$-ring, we have every commutative $QI$-ring is semisimple.
 
 Next, let $R$ be a $u$-$S$-semisimple ring. Then every $R$-module is $u$-$S$-injective by \cite[Theorem 3.5]{ZQ}. In particular, every $u$-$S$-quasi-injective $R$-module is $u$-$S$-injective. Thus $R$ is a $u$-$S$-$Q$$u$-$S$-$I$-ring. 
 
 Finally, let $R$ be a $u$-$S$-$Q$$u$-$S$-$I$-ring. Then every $u$-$S$-quasi-injective $R$-module is $u$-$S$-injective. Since every quasi-injective is $u$-$S$-quasi-injective by \cite[Remark 3.2 (2)]{MM}, then every quasi-injective $R$-module is $u$-$S$-injective. Hence $R$ is a $Q$$u$-$S$-$I$-ring.
\end{proof}

\begin{example}\label{ex1}
(1) The converse of the implication
\begin{center}
    semisimple rings $\Rightarrow$ $u$-$S$-semisimple rings
\end{center}is not true in general by \cite[Example 3.11]{ZQ}.\\[0.1cm]
(2) The converse of the implication \begin{center}
    $QI$-rings $\Rightarrow$ $Q$$u$-$S$-$I$-rings
\end{center} is not true in general. To see this, let $R$ and $S$ be as in \cite[Example 3.11]{ZQ}, then $R$ is a $u$-$S$-semisimple ring that is not a semisimple ring. So by Remark \ref{rem2}, $R$ is a $Q$$u$-$S$-$I$-ring that is not a $QI$-ring.
                
\end{example}

Let $S$ be a multiplicative subset of a ring $R$, $M$ an $R$-module, and $N$ a submodule of $M$. Recall that 

\begin{enumerate}
    \item[(1)] $N$ is called fully invariant in $M$ if $f(N) \subseteq N$ for every $f \in End_{R}(M)$ \cite{NY}.
    \item[(2)] $N$ is called a $u$-$S$-direct summand of $M$ if $M$ is $u$-$S$-isomorphic to $N\oplus N'$ for some $R$-module $N'$ \cite{KMOZ}.
\end{enumerate} 

Recall that an $R$-module $M$ is quasi-injective if and only if it is fully invariant in its injective envelope $E(M)$ \cite{AF}. The following result gives some characterizations of the $Q$$u$-$S$-$I$-rings.

\begin{thm}\label{thm1}
  Let $S$ be a multiplicative subset of a ring $R$. Then the following statements are equivalent:
   \begin{enumerate}
          \item[(1)] $R$ is a $Q$$u$-$S$-$I$-ring;
          \item[(2)] Every direct sum of two quasi-injective modules is $u$-$S$-quasi-injective;
          \item[(3)] Every fully invariant submodule of an injective module is a $u$-$S$-direct summand.
      \end{enumerate}
  \end{thm}

\begin{proof}
   $(1)\Rightarrow (2)$: Let $M$ and $N$ be two quasi-injective modules. Since $R$ is a $Q$$u$-$S$-$I$-ring, then $M$ and $N$ are $u$-$S$-injective and hence $M\oplus N$ is $u$-$S$-injective by \cite[Proposition 4.7 (1)]{QK}. Thus $M\oplus N$ is $u$-$S$-quasi-injective by \cite[Remark 3.2 (1)]{MM}. \\
   $(2)\Rightarrow (1)$: Let $M$ be a quasi-injective module. Let $f:M\to E$ be a monomorphism with $E$ injective. Then by (2), $M\oplus E$ is $u$-$S$-quasi-injective. Now, let $i_1:M\to M\oplus E$ and $i_2:E\to M\oplus E$ be the natural injections. If $p_1:M\oplus E\to M$ is the natural projection, then $p_1i_1=1_M$. Since $M\oplus E$ is $u$-$S$-quasi-injective and $M\xrightarrow{f} E\xrightarrow{i_2} M\oplus E$ is a monomorphism, then by Lemma \ref{lem1}, there is $g\in \text{End}_R(M\oplus E)$ such that $si_1=gi_2f$ for some $s\in S$. So $s1_M=sp_1i_1=p_1si_1=p_1gi_2f=f'f$, where $f':=p_1gi_2:E\to M$. Hence, the exact sequence $0\to M\xrightarrow{f} E\to \frac{E}{\text{Im}(f)}\to 0$ is $u$-$S$-split. By \cite[Lemma 2.8]{KMOZ}, $E$ is $u$-$S$-isomorphic to $M\oplus \frac{E}{\text{Im}(f)}$. But $E$ is $u$-$S$-injective, so by \cite[Proposition 4.7 (3)]{QK}, $M\oplus \frac{E}{\text{Im}(f)}$ is $u$-$S$-injective. Thus $M$ is $u$-$S$-injective by \cite[Corollary 2.8 (2)]{MM}. Therefore, $R$ is a $Q$$u$-$S$-$I$-ring.\\
$(1)\Rightarrow (3)$: Let $M$ be an injective module and $N$ be a fully invariant submodule of $M$. Let $f\in \text{End}_R(E(N))$. Since $M$ is injective, so $E(M)=M$. Let $i:E(N)\to M$ be the inclusion map. Since $M$ is injective, then there is $g\in \text{End}_R(M)$ such that the following diagram 
\[\xymatrix{
& M &\\
&E(N)\ar[u]^{if} \ar[r]^{i}&M\ar@{-->}[lu]_{g}
}\]
commutes. Since $N$ is fully invariant in $M$, then $g(N)\subseteq N$. So $f(N)=i(f(N))=g(N)\subseteq N$. Hence $N$ is fully invariant in $E(N)$ and thus $N$ is quasi-injective. By $(1)$, $N$ is $u$-$S$-injective. It follows that the exact sequence $0\to N\to M\to \frac{M}{N}\to 0$ is $u$-$S$-split. Thus, by \cite[Lemma 2.8]{KMOZ}, $M$ is $u$-$S$-isomorphic to $N\oplus \frac{M}{N}$. Therefore, $N$ is a $u$-$S$-direct summand of $M$.\\
$(3)\Rightarrow (1)$:  Let $M$ be a quasi-injective module. Then $M$ is fully invariant in $E(M)$. By $(3)$, $M$ is a $u$-$S$-direct summand of $E(M)$. Since $E(M)$ is $u$-$S$-injective, then $M$ is $u$-$S$-injective by \cite[Proposition 4.7 (3)]{QK} and \cite[Corollary 2.8 (2)]{MM}. Thus $R$ is a $Q$$u$-$S$-$I$-ring.
\end{proof}

The following result gives a characterization of the $u$-$S$-$Q$$u$-$S$-$I$-rings.

\begin{thm}\label{thrm1} 
   Let $S$ be a multiplicative subset of a ring $R$. Then the following statements are equivalent:
    \begin{enumerate}
          \item[(1)] $R$ is a $u$-$S$-$Q$$u$-$S$-$I$-ring;
          \item[(2)] Every direct sum of two $u$-$S$-quasi-injective modules is $u$-$S$-quasi-injective.
      \end{enumerate}
\end{thm}

\begin{proof}
The proof is similar to the proof of $(1)\Leftrightarrow (2)$ in Theorem \ref{thm1}. 
\end{proof}

\begin{corollary}
    Let $S$ be a multiplicative subset of a ring $R$ and $A,B$ be $R$-modules. Let $R$ be a $u$-$S$-$Q$$u$-$S$-$I$-ring. Then $A\oplus B$ is $u$-$S$-quasi-injective if and only if $A$ and $B$ are $u$-$S$-quasi-injective. 
\end{corollary}

\begin{proof}
    ($\Rightarrow$). This follows from \cite[Proposition 3.8]{MM}. \\
    ($\Leftarrow$). This follows from Theorem \ref{thrm1}.
\end{proof}

Let $\mathfrak{p}$ be a prime ideal of a ring $R$. We say that $R$ is a $Q$$u$-$\mathfrak{p}$-$I$-ring ($u$-$\mathfrak{p}$-$Q$$u$-$\mathfrak{p}$-$I$-ring) if $R$ is a $Q$$u$-$S$-$I$-ring ($u$-$S$-$Q$$u$-$S$-$I$-ring), where $S=R\setminus \mathfrak{p}$. The following proposition gives a local characterization of the $QI$-rings. 

\begin{proposition}\label{plocal}
 Let $R$ be a ring. Then the following statements are equivalent:
     \begin{enumerate}
          \item[(1)] $R$ is a $QI$-ring;
          \item[(2)] $R$ is a $Q$$u$-$\mathfrak{p}$-$I$-ring for every $\mathfrak{p}\in \text{Spec}(R)$;
          \item[(3)] $R$ is a $Q$$u$-$\mathfrak{m}$-$I$-ring for every $\mathfrak{m}\in \text{Max}(R)$;
          \item[(4)] $R$ is a semisimple ring.
      \end{enumerate}

\end{proposition}

\begin{proof}
    $(1)\Rightarrow (2)$ and $(1)\Leftrightarrow (4)$: Follow from Remark \ref{rem2}.\\
     $(2)\Rightarrow (3)$: Clear.\\
    $(3)\Rightarrow (1)$: Let $M$ be a quasi-injective module. Then by (3), $M$ is $u$-$\mathfrak{m}$-injective for every $\mathfrak{m}\in \text{Max}(R)$. Thus, by \cite[Proposition 4.8]{QK}, $M$ is injective. Therefore, $R$ is a $QI$-ring.
\end{proof}

\begin{corollary}
    Let $R$ be a ring. Then the following statements are equivalent:
     \begin{enumerate}
          \item[(1)] $R$ is a $QI$-ring;
          \item[(2)] $R$ is a $u$-$\mathfrak{p}$-$Q$$u$-$\mathfrak{p}$-$I$-ring for every $\mathfrak{p}\in \text{Spec}(R)$;
          \item[(3)] $R$ is a $u$-$\mathfrak{m}$-$Q$$u$-$\mathfrak{m}$-$I$-ring for every $\mathfrak{m}\in \text{Max}(R)$;
           \item[(4)] $R$ is a semisimple ring.
      \end{enumerate}
\end{corollary}
    
\begin{proof}
$(1)\Rightarrow (2)$ and $(1)\Leftrightarrow (4)$: Follow from Remark \ref{rem2}.\\
     $(2)\Rightarrow (3)$: Clear.\\
    $(3)\Rightarrow (1)$: Suppose that $R$ is a $u$-$\mathfrak{m}$-$Q$$u$-$\mathfrak{m}$-$I$-ring for every $\mathfrak{m}\in \text{Max}(R)$, then by Remark \ref{rem2}, we have $R$ is a $Q$$u$-$\mathfrak{m}$-$I$-ring for every $\mathfrak{m}\in \text{Max}(R)$. Thus, by Proposition \ref{plocal}, $R$ is a $QI$-ring.
\end{proof}

The last result of this section gives a characterization of rings in which every $u$-$S$-pseudo-injective module is $u$-$S$-injective.
    
\begin{thm}\label{thm3} 
    Let $R$ be a ring and $S$ a multiplicative subset of $R$. Then the following statements are equivalent:
    \begin{enumerate}
          \item[(1)] Every $u$-$S$-pseudo-injective module is $u$-$S$-injective;
          \item[(2)] Every direct sum of two $u$-$S$-pseudo-injective modules is $u$-$S$-pseudo-injective.
      \end{enumerate}
\end{thm}

\begin{proof}
$(1)\Rightarrow (2)$: Let $M$ and $N$ be two $u$-$S$-pseudo-injective modules. Then by (1), $M$ and $N$ are $u$-$S$-injective and hence $M\oplus N$ is $u$-$S$-injective. Thus, by Remark \ref{rem1} (2), $M\oplus N$ is $u$-$S$-pseudo-injective. \\
   $(2)\Rightarrow (1)$: Let $M$ be a $u$-$S$-pseudo-injective module. Then by (2), $M\oplus M$ is $u$-$S$-pseudo-injective. So by Corollary \ref{cor1}, $M$ is $u$-$S$-quasi-injective. Hence, every $u$-$S$-pseudo-injective module is $u$-$S$-quasi-injective ...(*). If $M$ and $N$ are two $u$-$S$-quasi-injective modules, they are $u$-$S$-pseudo-injective by Remark \ref{rem1} (2), so by (2), $M\oplus N$ is $u$-$S$-pseudo-injective and hence by (*), $M\oplus N$ is $u$-$S$-quasi-injective. Thus, every direct sum of two $u$-$S$-quasi-injective modules is $u$-$S$-quasi-injective. Therefore, by (*) and Theorem \ref{thrm1}, $(1)$ holds. 
   \end{proof}

We end this paper by listing the following unanswered questions:

\begin{question}
    Let $R$ be a commutative ring and $M$ a pseudo-injective $R$-module that is not quasi-injective. Is it true that $M$ is $u$-$\mathfrak{m}$-pseudo-injective for every $\mathfrak{m}\in \text{Max}(R)$?
\end{question}

\begin{question}
      Let $R$ be a commutative ring and $S$ a multiplicative subset of $R$. Is it true that $R$ is $u$-$S$-semisimple if and only if $R$ is $u$-$S$-$Q$$u$-$S$-$I$ if and only if $R$ is $Q$$u$-$S$-$I$?
\end{question}

\textbf{Conflict of interest:} The authors declare that they have no conflict of interest.\\

\textbf{Acknowledgments:} The authors would like to thank the reviewers and editor for their valuable comments that improved the quality of the paper.

\end{document}